\documentclass[11pt]{amsart}

\usepackage{amsthm}
\usepackage{amsmath}
\usepackage{amssymb}
\usepackage[margin=1.4in]{geometry}
\usepackage{enumerate}
\usepackage{hyperref}
\usepackage{mathrsfs}
\usepackage{color}
\usepackage{bm}

\usepackage{microtype}
\usepackage{graphicx}
\usepackage{pdfpages}
\usepackage{todonotes}
\usepackage{enumerate}

\title{Diffusion limit for a slow-fast standard map}
\author[Alex Blumenthal]{Alex Blumenthal$^\dagger$}
\thanks{$\dagger$ This material is based upon work supported by the National Science Foundation under Award No. DMS-1604805.}
\address{Alex Blumenthal\\
Department of Mathematics\\
University of Maryland\\
4417 Mathematics Bldg,  College Park,  MD 20742, USA}

\email{{\tt alexb123@math.umd.edu}}
\urladdr{\href{http://www.math.umd.edu/~alexb123}{http://www.math.umd.edu/~alexb123}}

\author[Jacopo De Simoi]{Jacopo De Simoi$^{\ddagger}$}
\address{Jacopo De Simoi\\
Department of Mathematics\\
University of Toronto\\
40 St George St. Toronto, ON, Canada M5S 2E4}
\email{{\tt jacopods@math.utoronto.ca}}
\urladdr{\href{http://www.math.utoronto.ca/jacopods}{http://www.math.utoronto.ca/jacopods}}
\thanks{$\ddagger$ J.D.S. is supported by the NSERC Discovery grant,
  reference number 502617-2017}

\author[Ke Zhang]{Ke Zhang$^*$}
\thanks{$*$ K.Z. is supported by the NSERC Discovery grant, reference number 436169-2013.}
\address{Ke Zhang\\
Department of Mathematics\\
University of Toronto\\
40 St George St. Toronto, ON, Canada M5S 2E4}
\email{{\tt kzhang@math.utoronto.ca}}
\urladdr{\href{http://www.math.utoronto.ca/kzhang/}{http://www.math.utoronto.ca/kzhang/}}

\date{\today}

\theoremstyle{theorem}
\newtheorem{thm}{Theorem}
\newtheorem{thmA}{Theorem}
\newtheorem{cor}[thm]{Corollary}

\newtheorem{lem}[thm]{Lemma}

\newtheorem{prop}[thm]{Proposition}

\theoremstyle{definition}

\newtheorem{defn}[thm]{Definition}
\newtheorem{rmk}[thm]{Remark}

\newtheorem*{nrmk}{Notational remark}

\newcommand{\E}{\mathbb{E}}

\newcommand{\N}{\mathbb{N}}
\renewcommand{\P}{\mathbb{P}}
\newcommand{\R}{\mathbb{R}}

\newcommand{\Z}{\mathbb{Z}}

\newcommand{\Uc}{\mathcal{U}}
\newcommand{\Bc}{\mathcal{B}}
\newcommand{\Fc}{\mathcal{F}}

\newcommand{\Lc}{\mathcal{L}}

\newcommand{\Sc}{\mathcal{S}}

\newcommand{\Nc}{\mathcal{N}}

\newcommand{\Ec}{\mathcal E}

\newcommand{\Jc}{\mathcal J}

\newcommand{\bE}{\mathbb{E}}

\newcommand{\cL}{\mathcal{L}}

\renewcommand{\a}{\alpha}
\renewcommand{\b}{\beta}

\newcommand{\e}{\epsilon}

\newcommand{\graph}{\operatorname{graph}}

\newcommand{\ifrac}[2]{#1/#2}

\newcommand{\T}{\mathbb T}
\newcommand{\Cc}{\mathcal C}
\newcommand{\Leb}{\operatorname{Leb}}

\renewcommand{\graph}{\operatorname{graph}}

\newcommand{\Pc}{\mathcal P}
\newcommand{\Ic}{\mathcal I}

\newcommand{\Bor}{\operatorname{Bor}}

\newcommand{\modone}{\, \, (\operatorname{mod} \, 1)}

\newcommand{\leb}{\Leb}

\newcommand{\bmat}[1]{\begin{bmatrix} #1 \end{bmatrix}}
\newcommand{\floor}[1]{\lfloor #1\rfloor}

\newcommand{\eg}{e.g.\ }
\newcommand{\ie}{i.e.\ }

\newcommand{\notej}[1]{\todo[inline]{Jacopo: #1}{}}
\newcommand{\textr}[1]{{\color{red}#1}}
\newcommand{\ignore}[1]{}

\begin{document}

\maketitle

\begin{abstract}
Consider the map
$(x, z) \mapsto (x + \epsilon^{-\alpha} \sin (2\pi x) +
\epsilon^{-(1+\alpha)}z, z + \epsilon \sin(2\pi x))$, which is conjugate
to the Chirikov standard map with a large parameter. The parameter
value $\alpha = 1$ is related to ``scattering by resonance''
phenomena. For suitable $\a$, we obtain a central limit theorem for
the slow variable $z$ for a (Lebesgue) random initial condition. The
result is proved by conjugating to the Chirikov standard map and
utilizing the formalism of standard pairs. Our techniques also yield
for the Chirikov standard map a related limit theorem and a
``finite-time'' decay of correlations result.
\end{abstract}

\section{Introduction and statement of results}

\subsection{The slow-fast standard map}

Throughout, $\a > 0$ is fixed. We consider the discrete-time slow fast
system $G_\e$ on the cylinder $ \T^1 \times \R$ defined as follows:
\[
G_\e(x,z) = (x + \e^{- \a} \sin(2 \pi x) + \e^{- (1 + \a)} z \modone,
z + \e \sin (2 \pi x))
\]
This map is a composition of two maps $G_\e = S_\e \circ T_\e$, where
the `tilt' map $T_\e$ and the `shear' map $S_\e$ are defined by
\[
T_\e (x,z) = (x, z + \e \sin (2\pi x) ) \, , \quad S_\e(x, z) = (x + \e^{- (1 + \a)} z \modone , z) \, .
\]
This combination of tilting and shearing serves as a good model on
many slow-fast physical systems: see the discussion in Section \ref{subsec:discussionModel} below.

%


\subsection{Statement of results}


Writing $(x^\e_n , z^\e_n) = G_\e^n(x_0, z_0)$ for $(x_0, z_0) \in \T^1 \times \R$ fixed, observe that
\begin{align*}
z^\e_N = z_0 + \e \sum_{n = 0}^{N-1} \sin (2\pi x^\e_n) \, .
\end{align*}

The $x$ coordinate is clearly `fast' relative to the $z$, and so one
anticipates $z^\e_N$ to have a \emph{diffusion limit} in the regime
$N = N(\e) = \lfloor \e^{-2} \rfloor$, when we consider it as a random
variable with respect to the initial conditions $(x_{0}, z_{0})$.  This
does not follow from conventional averaging arguments, however, since,
as will be explained in detail in the following sections, the
fast dynamics has critical behavior at
$x \approx \frac14, \frac34$ (the zeros of
$x \mapsto 1 + 2 \pi \e^{- \a} \cos(2 \pi x)$).

Our approach is to conjugate the above system to the Standard Map:
\begin{align}\label{e_Standard Map definition}
F_L (x, y) = (x + y + L \sin (2\pi x), y + L \sin (2\pi x))
\end{align}
by the change of variables $z = \e^{ 1 + \a} y$; here, the parameter
$L$ is defined by $L = \e^{- \a}$. Notice that the $x$ coordinate is
unchanged, and so we have that $x^\e_n$ is the $x$-{}coordinate of
$F^n_L(x_0, y_0)$ where $y_0 := \e^{- (1 + \a)} z_0$. Thus, the diffusion
limit for $z_{N(\e)}^\e$ above is equivalent to a central limit theorem
for the sequence
\begin{align}\label{stdMapForm}
\bigg( \frac{1}{\sqrt N} \sum_{i = 0}^{N-1} \psi \circ F^i_L \bigg)_{L \to \infty} \, .
\end{align}
Here, $\psi(x, y) = \sin (2\pi x)$, and in the above sequence, we write
$N = N(L) = \lfloor L^{\beta} \rfloor$, $\beta := 2/ \a$; this scaling
is equivalent to the original diffusion limit for $z_{N(\e)}^\e$.



\begin{thmA}\label{thm:diffusion}
Suppose $\alpha > 8$ and let $[a, b] \subset \R$ be a non-trivial
interval. Let $X, Z$ be uniformly distributed random variables on $\T$
and $[a, b]$ respectively. Define
$Z_\epsilon^n = \pi_z G_\epsilon^n(X, Z)$, then for
$N(\epsilon) = \floor{\epsilon^{-2}}$, the random variable
\begin{align*}
Z^{N(\epsilon)}_\epsilon - Z
\end{align*}
converges in distribution to the centered Gaussian $\Nc(0, \frac12)$,
as $\epsilon \to 0$.
\end{thmA}


Theorem~\ref{thm:diffusion} will be deduced from the following analogous
result for Standard Maps. In the following results, we regard $F_L$ as a
diffeomorphism of $\T^2 \cong \R^2 / \Z^2$.
Let $X, Y$ be independent random variables distributed uniformly on
$\T^1$.
\begin{thmA}\label{thm:CLT}
  Let $\phi : \T^1 \to \R$ be a $C^1$ observable, regarded as an
  $x$-dependent observable on $\T^2$. Assume $\int \phi \, dx = 0$ and
  that $\phi$ is not identically $0$. Let $N : \R_{> 0} \to \N$ be an
  increasing function and assume
\begin{equation}
  N(L) \cdot L^{- \frac{1}{4}} \to 0 \quad \quad \text{ as } L \to \infty \, .
\end{equation}
Then,
\[
  \frac{1}{\sqrt{N(L)}} \sum_{i = 0}^{N(L)-1} \phi \circ F^i_L(X, Y)
\]
converges in distribution to the centered Gaussian
$\Nc(0, \sigma^2)$ with variance $\sigma^2 = \int \phi^2 dx > 0$.
\end{thmA}

\bigskip

As a consequence of our techniques we obtain the following result on
decay of correlations, which we report here as
a potentially useful result in its own right.
\begin{thmA}\label{thm:decayLebesgue}
There exists a constant $C > 0$ for which the following holds for all
$L > 0$ sufficiently large.  Let $\phi, \psi : \T^1 \to \R$ be $C^1$
observables, each regarded as $x$-dependent observables on
$\T^2$. Then, for all $n \geq 1$:
\begin{align*}
\bigg| \int \psi \cdot \phi \circ F^n_L - \int \phi \int \psi \bigg| \leq C \| \phi \|_{C^1} \| \psi \|_{C^1} \bigg( (n-1) L^{- 3/4} + L^{- 1/2} \bigg).
\end{align*}
\end{thmA}

\subsection{Discussion of the model and previous work} \label{subsec:discussionModel}

Our study of the system $G_\epsilon$ is primarily motivated by the following model.

\subsubsection*{Scattering by resonance}
We only give a heuristic picture here and refer to \cite{Nei1975a},
\cite{Nei1990}, \cite{Nei1996}, \cite{Dol2012} for details. To use a
specific example (see \cite{Dol2012}), consider the slow-fast system
\[
\dot{\phi} = f(\phi, I, \theta, \epsilon), \quad \dot{I} = g(\phi, I, \theta, \epsilon), \quad
\dot{\theta} = \epsilon^{-2} \omega(\phi, I, \epsilon), \quad \phi, \theta \in \T^1, \,  I \in \R.
\]
It is assumed that the averaged system
\[
\dot{\phi} = \int_0^1 f(\phi, I, \theta, 0) d\theta = p(I), \quad \dot{I} = \int_0^1 g(\phi, I, \theta, 0) d\theta = 0
\]
is completely integrable.  However, the averaging is not justified near
the resonant surface $\{\omega(\phi, I, 0) = 0\}$, since the fast
variable $\theta$ is no longer fast.

As the orbit in $(I, \phi)$ passes through the resonances, two
different phenomena may happen:
\begin{itemize}
	\item \emph{Strong resonance}, where there is a probability of
	$O(\epsilon)$ for the orbit to be \emph{captured} by the resonance,
	and stay captured for a random time before it is \emph{repelled}. See
	\cite{Dol2012} for a full analysis of this picture and the related
	limit theorems.

	\item \emph{Weak resonance}, where the orbit passes through the
	resonance without being captured.  After the passing the variable $I$
	changes by order $\e$, with average flux $0$. This is called
	\emph{scattering by resonance}.
\end{itemize}

As the orbit crosses a weak resonance, the change to the variables
$(\theta, I)$ can be approximated by a map of the type $T_\e$; while the
``free flight'' between two crossings of the resonance is approximated
by the map $S_\e(s, z) = (x + \e^{-2}z, z)$. As a result, successive
passages through weak resonances can be modeled by sequential
applications of maps of type $G_\e$ (with $\alpha = 1$).

\smallskip

Theorem \ref{thm:diffusion} of this paper does not apply to the $\alpha = 1$ regime
described above: indeed, to take $\alpha$ smaller in Theorem \ref{thm:diffusion} would necessitate
controlling the dynamics of the Standard map $F_L$, in the sense of Theorem \ref{thm:decayLebesgue},
 for timescales far longer than $N \sim L^{1/4}$. 


\subsubsection*{The Standard Map}

The Standard Map is a one-parameter family of area-preserving analytic
diffeomorphisms of $\T^{2}$.  It has been the subject of extensive
numerical and analytical study, starting from the pioneering work of
V. Chirikov and (independently) J. B. Taylor.  From the physical point
of view it describes the dynamics of a mechanical system known as the
``kicked rotor'', but it can be found in a large number of different other
models.  For example: it describes ground states of the
Frenkel--Kontorova Model (see~\cite{MR0001169,Aubry}); it models
dynamics of particles in accelerators (see \cite{MR536429,MR602110}) and
dynamics of balls bouncing on periodically oscillating platform (see \eg
\cite{MR0364764,MR2533973}); and can be regarded as a
toy model for stretching and folding mechanisms in fluid mechanics
(see \eg \cite{crisanti1991lagrangian}).

From the mathematical point of view it has been studied as a natural
example of dynamical system exhibiting \emph{mixed behavior}: it is
conjectured that the phase space of the standard map has positive
Lebesgue measure sets where the dynamics is hyperbolic and enjoys strong
statistical properties (``stochastic sea'') and positive Lebesgue
measure sets where the dynamics is regular (elliptic islands)
\cite{duarte1994plenty}.  In this respect, points belonging to the
hyperbolic component of the phase space should undergo some sort of
diffusion.  However, this fact has notoriously eluded rigorous proof for
many years and is widely believed to be astonishingly difficult to
prove. The strongest positive partial results are those of Gorodetski
\cite{gorodetski2012stochastic}, who proved that the hyperbolic set for
the standard map has Hausdorff dimension 2 for a ``large'' set of
sufficiently large $L$, and Berger and Turaev \cite{berger2017herman},
who proved that the standard map is $C^r$ ($r \geq 2$) close to a
volume-preserving map with positive metric entropy.

A natural problem of intermediate difficulty, pursued in the present
article, is to
consider a scaling limit in which the natural parameter of the Standard
Family is increased together with the number of iterations.  A first
result about statistical properties of the Standard Family in this
scenario can be found in \cite{Alex}, in which it is shown that
compositions of standard maps with increasing parameter exhibit both
asymptotic decay of correlations and a Central Limit Theorem with
respect to Holder-continuous observables. A correlation estimate
analogous to that in Theorem \ref{thm:decayLebesgue} is also exhibited.
While both the present article and \cite{Alex} share some features in
common (e.g., a reliance on correlations estimates for standard pairs),
the two implementations are distinct.  A key difference is that the
correlation estimate in Theorem \ref{thm:decayLebesgue} is much stronger
than the one appearing in \cite{Alex}, but at the same time takes
advantage of the simplifying assumption of working only for
$x$-dependent observables, whereas the results of \cite{Alex} apply to
all Holder-continuous observables. This difference also means that the
techniques used in the present manuscript differ significantly from
those in \cite{Alex}.

\subsubsection*{Background on the proof: Standard pairs}
Standard pairs are a modern tool which can be used to study statistical
properties of systems with some hyperbolicity.  They have been
introduced by Dolgopyat in a variety of settings (see for example
\cite{Dolgopyat2004}, \cite{Dolgopyat2004a}, \cite{Dol2012})
and have proved to be of invaluable help.  In a nutshell,
standard pairs are probability measures on the phase space which enjoy
particularly good dynamical properties (see
Lemmata~\ref{lem:cut-standard}--\ref{lem:cut-full}).  The main feature
of such measures is that they allow to introduce a sensible notion of
\emph{conditioning} in the deterministic setting.  In probability,
conditioning is one of the most basic and useful techniques, and one
would like to employ this tool also in our situation.  Clearly, in
deterministic settings, some care must be taken, as if one were to
condition on the configuration of the system at any given time, the
whole probabilistic picture would collapse (as no randomness would be
present anymore).  Standard pairs provide a very efficient solution to
this fundamental problem.


\subsubsection*{Plan for the paper}

The plan for the paper is as follows. In Section 2 we give some
preliminaries, including the definition of standard pair and various
related notions used in this paper.  In Section 3 we consider the
dynamics of standard pairs, prove results on correlation decay for
standard pairs, and use these to prove Theorem \ref{thm:decayLebesgue}. In
Section 4 we prove the Central Limit Theorem as stated in Theorem
\ref{thm:CLT}. In Section 5 we deduce Theorem \ref{thm:diffusion} from
Theorem~\ref{thm:CLT}.


  \subsubsection*{Notation and conventions}
\begin{itemize}
	\item We parametrize the circle $\T^1$ by the half-open interval
	$[0,1)$. Additive formulas in $\T^1$ are always considered$\modone$,
	i.e., under the natural projection
	$\R \to \T^1 = \R / \Z \cong [0,1)$. We parametrize $\T^2$ by
	$[0,1)^2$.
	\item We call a continuous observable $\phi : \T^2 \to \R$ {\bf x-dependent} if it can be represented as $\phi(x, y) = \hat \phi(x)$ for some $\hat \phi : \T^1 \to \R$. In this manuscript we will often use the same notation
	$\phi$ for both the observable on $\T^1$ and the corresponding
	$x$-dependent observable on $\T^2$.
	\item For a $C^1$ function $g$ defined on an open interval in $\R$ or $\T^1$, we write $\dot g$ or $\frac{d}{dx} g$
	for the derivative of $g$.  We write $\| g \|_{C^0}$ for the uniform norm of $g$ and $\| g \|_{C^1} = \max\{ \| g \|_{C^0} , \| \dot g \|_{C^0} \} $.
	\item $\Leb$ refers to normalized Lebesgue measure on $\T^2 \cong [0,1)^2$.
	\item Let $G = G(L)$ denote any quantity depending on the parameter $L$. We say that another quantity $H = H(L)$ is in the class $O(G)$, written $H = O(G)$, if $\limsup_{L \to \infty} \frac{|H(L)|}{|G(L)|} < \infty$. We say $H$ is in the class $o(G)$, written $H = o(G)$, if $\lim_{L \to \infty} \frac{|H(L)|}{|G(L)|} = 0$.
	\item We write $G \approx H$ if $G/H = O(1)$ and $H/G = O(1)$.
\end{itemize}

\section{Preliminaries}

\subsection{Coordinate change}

Under the coordinate change $y \mapsto x - y$, the Standard map $F_L$
(defined in~\eqref{e_Standard Map definition})is
conjugate to the map
\[
\hat F_L(x, y) = (2x - y + L \sin (2 \pi x) \modone, x) \, ,
\]
which we regard as a map on $\T^2$. This change in the $y$-coordinate
has no effect on the analysis of our diffusion limit, since the
observable $\phi$ is $x$-dependent. This form for the Standard Map is
convenient and will be used from now on. Hereafter we abuse notation and
write $F = \hat F_L$, dropping the subscript $L$ (which is implicit
throughout).  Additionally, we define
\begin{align*}
f = f_L : \T^1 \to \R \, , \quad f(x) := 2x + L \sin(2 \pi x) \, ,
\end{align*}
so that $F = F_L$ has the form
\begin{align*}
F(x, y) = (f(x) - y \modone, x) \, .
\end{align*}
In all that follows, we regard $F$ as a map on the torus $\T^2 = \T^1 \times \T^1$. At times, it is also
convenient to use instead the map $\tilde F : \T^2 \to \R \times \T^1$ obtained by omitting
the ``$\modone$'' in the $x$-coordinate.

\subsection{Predominant hyperbolicity of $F$}

For fixed $\eta \in (0,1)$, define
\[
\Sc_\eta = \{ (x,y) \in \T^2 :  |2 + 2 \pi L \cos(2 \pi x)| \leq 2 L^\eta \} = \Bc_\eta \times \T^1 
\]
For all $L$ large and any $\eta \in (0,1)$, the set $\Sc_\eta$ consists
of two small, disjoint vertical strips in $\T^2$; observe that,
trivially, $S_{\eta}\subset S_{\eta'}$ for $\eta < \eta'$.  Away from
the set $\Sc_\eta$, the map $F$ is strongly expanding in the horizontal
direction to order $L^\eta$; for this reason we refer to the $\Sc_\eta$
as {\bf critical strips}.

To make this picture more precise, for $\xi > 0$ let us define the {\bf
horizontal cone}
\begin{align*}
\Cc_\xi = \{ v  = (u, w) \in \R^2 : |w| \leq \xi |u| \} \, .
\end{align*}



\begin{lem}\label{lem:predomHyp} For all $L$ sufficiently large, the
following holds for each $\eta \in (0,1)$.
\begin{enumerate}[(a)]
\item The set $\Bc_\eta$ is the union of two disjoint intervals,
each of length $\approx L^{-1 + \eta}$, containing respectively
the points $1/4$ and $3/4$. In particular, the set $\Sc_\eta$
satisfies
      \begin{align*}
      \Leb(\Sc_\eta) = O(L^{-1 + \eta}) \, .
      \end{align*}
      \item Let $\xi \leq  L^\eta$. For all
      $p = (x,y) \in \T^2 \setminus \Sc_\eta$, we have that
      \begin{align*}
      dF_p \Cc_\xi \subset \Cc_{\xi'}
      \end{align*}
      for any $\xi' \geq \frac{1}{2L^\eta - \xi}$.
  \end{enumerate}
  \end{lem}
  \begin{proof}
  (a) Since $S_\eta = \{|2 L^{-1} + 2\pi \cos(2\pi x)| \le 2 L^{-1 + \eta}\} \subset \{2\pi |\cos{2\pi x}| \le 4 L^{-1 + \eta}\}$, the estimate follows easily.

  (b) Note $|\dot f(x)| > 2L^\eta$ for $x \notin \Bc_\eta$. For a tangent vector $(1, m) \in C_\xi$, we have
  \[
  	\begin{aligned}
      dF_p \bmat{1 \\ m}  & = \bmat{ 2 + 2\pi L \cos(2 \pi x) & 1 \\ -1 & 0  }\bmat{ 1 \\ m}
  = \bmat{\dot f(x) + m \\  - 1} \\
  & =  (\dot f(x) + m) \bmat{1 \\ \frac{-1}{\dot f(x) + m}}
    \in C_{\xi'}. & \qedhere
  	\end{aligned}
  \]
  \end{proof}

  For our purposes, we usually work with the cone $\Cc_{1/10}$, which by
  Lemma \ref{lem:predomHyp} is mapped into itself away from $\Sc_{1/4}$, if $L$ is sufficiently large.

%
%

\subsection{u-curves}

We work mostly with $C^2$ curves, the tangents to which lie in the cone $\Cc_{1/10}$. More precisely:
\begin{defn}
Let $\gamma$ be a $C^2$ embedded curve in $\T^2$. We say that $\gamma$ is a {\bf u-curve} if $\gamma = \{ (x, h_\gamma(x) ) : x \in I_\gamma\}$, where
\begin{itemize}
	\item[(a)] $I_\gamma \subsetneq \T^1$ is an open interval; and
	\item[(b)] $h_\gamma : I_\gamma \to \T^1$ is a $C^2$ mapping with $\| \dot h_\gamma\|_{C^0} \leq 1/10$, $\| \ddot h_\gamma\|_{C^0} \leq L$.
\end{itemize}
The {\bf length} of a u-curve $\gamma$ is defined (with a small abuse of
terminology) as the length of the interval $I_\gamma$.  We call $\gamma$
a {\bf fully-crossing u-curve} if $I_\gamma = (0,1)$.
\end{defn}

Away from the critical strips, u-curves map to u-curves, for which the
following lemma is useful.

\begin{lem}\label{lem:ucurve}
Fix $\eta \in [1/4, 1)$. Let $\gamma$ be a u-curve with
$\gamma \cap S_\eta = \emptyset$.  Then,
$\tilde \gamma := \tilde F(\gamma)$ is a $C^2$ curve of the form
$\tilde \gamma = \{ (x, \tilde h(x)) : x \in \tilde I\}$, where
$\tilde I \subset \R$ is an interval and
$\tilde h : \tilde I \to \T^2$ is a $C^2$ mapping with
$\| \frac{d}{dx} \tilde h\| \le 1/10$, $\| \frac{d^2}{dx^2} \tilde h\| \leq
L$.
\end{lem}
From Lemma \ref{lem:ucurve}, we can represent $F(\gamma)$
as a finite union of u-curves by subdividing $\tilde \gamma$ into u-curves of length $< 1$
and then projecting $\R \times \T^1 \to \T^2$.

\begin{proof}[Proof of Lemma \ref{lem:ucurve}]
Define $f_\gamma : I_\gamma \to \R$ by setting
\[
f_\gamma(x) = f(x) - h_\gamma(x) \, .
\]
Since $\dot{f}_\gamma(x) = 2 + 2\pi L \cos(2\pi x) - \dot{h}_\gamma(x)$, we have the estimate $|\dot{f}_\gamma(x)| \ge 2L^\eta - |\dot{h}_\gamma| \ge L^\eta$, which will be useful throughout. Let us also note $|\ddot{f}_\gamma(x)| = |4\pi^2 L \sin(2\pi x) + \ddot{h}_\gamma(x)| = O(L)$.

As one can check, $\tilde F(x, h_\gamma(x)) = (f_\gamma(x), x)$,
from which Lemma \ref{lem:ucurve} follows with
$\tilde h := (f_\gamma)^{-1} : \tilde I \to \T^1$, where
$\tilde I : = f_\gamma(I)$. The estimates on
$\| \frac{d}{dx} \tilde h\|, \| \frac{d^2}{dx^2} \tilde h\|$
immediately follow from the formulae
\begin{align*}
\frac{d}{dx} \tilde h &= \frac{1}{\dot f_\gamma} \circ f_\gamma^{-1}, &\frac{d^2}{dx^2} \tilde h &= - \frac{\ddot f_\gamma}{( \dot f_\gamma)^3} \circ f_\gamma^{-1} \, . \qedhere
\end{align*}
\end{proof}
\subsection{Standard pairs}
Let $a_0 \in (0,1/8]$.
\begin{defn}
  A {\bf measure pair} is a pair $(\gamma, \rho)$, where $\gamma$ is a
  u-curve and $\rho : I_\gamma \to (0,\infty)$ is a nonvanishing $C^1$
  probability density on $I_\gamma$ (in particular,
  $\int_{I_\gamma} \rho dx = 1$).  We distinguish three subclasses of
  measure pairs:
\begin{enumerate}[(a)]
\item We call $(\gamma, \rho)$ a {\bf standard pair} if (i)
$|I_\gamma| > a_0$, and (ii) $\rho$ satisfies the distortion estimate
\begin{align}\label{eq:dist}
\left\|\frac{d \log \rho}{dx}\right\| &\leq 3 C_0 \, ,
\end{align}
where $C_0 = 8 \pi^2$, $\|\cdot\|$ denotes the uniform norm and
$a_0 > 0$ is a small, fixed positive constant (see above). We call
$(\gamma, \rho)$ a {\bf fully-crossing standard pair} if $\gamma$ is
fully-crossing.

\item We call $(\gamma, \rho)$ a {\bf substandard pair} if (i)
$|I_\gamma| \in [L^{- \frac12}, a_0]$; (ii) $\rho$ satisfies
$\| \frac{d \log \rho}{dx} \| \leq 2 C_0 L^{\frac12}$; and (iii)
$I_\gamma \cap \Bc_{1/2} = \emptyset$ (equivalently,
$\gamma \cap \Sc_{1/2} = \emptyset$).
\end{enumerate}
\end{defn}

\begin{rmk}
The value $a_0 \in (0,1/8]$ above is fixed and independent of $L$,
although for our purposes it will be useful fix it at a sufficiently
small value. This will be done by the end of Section 3 (see Remark
\ref{rmk:parameter}).  Before then, however, we include the parameter
$a_0$ in our $O(\cdots)$ estimates.
\end{rmk}

Moreover, for a curve $\gamma$ we write $\Leb_\gamma$ for the
(un-normalized) Lebesgue measure on $\gamma$. Since
$x \mapsto (x, h_\gamma(x))$ is a diffeomorphism of $I_\gamma$ onto
$\gamma$, we identify $\Leb_\gamma$ with the corresponding measure on
$I_\gamma$ given by
\[
 d \Leb_\gamma(x, h_\gamma(x)) = \sqrt{1 + \dot h_\gamma^2(x) } \, dx \, .
\]

\medskip


Some additional conventions: we regard measure pairs $(\gamma, \rho)$ as
measures on the curve $\gamma$ itself via the parametrization
$x \mapsto (x, h_\gamma(x))$. In particular, $F_*(\gamma, \rho)$ refers
to the pushforward measure of $(\gamma, \rho)$ on the image set
$F(\gamma)$ (which, we note, need not be a u-curve). Moreover, for
continuous observables $\phi : \T^2 \to \R$ we write
$\int \phi \, d (\gamma, \rho)$ for the integral of $\phi$ with respect
to the measure $(\gamma, \rho)$ on $\gamma$.

\medskip

Before proceeding, we record the following distortion estimate, which
will be used many times in the coming proofs.

\begin{lem}\label{lem:distStdPair}
Let $\eta \in [1/4, 1]$.  Fix u-curves $\gamma, \gamma' \subset \T^2$
for which $\gamma \cap \Sc_\eta = \emptyset$ and
$F(\gamma) \supset \gamma'$.  Let $(\gamma, \rho)$ be a measure pair,
and define $\rho'$ so that $(\gamma', \rho')$ is the normalization of
$F_*(\gamma, \rho) |_{\gamma'}$.  Then,
\begin{align*}
\bigg\| \frac{d \log \rho' }{dx} \bigg\| \leq L^{- \eta} \bigg\|
\frac{d \log \rho}{dx} \bigg\| + C_0 L^{1 - 2 \eta} \, ,
\end{align*}
where $C_0 := 8 \pi^2$ and $\| \cdot \|$ refers to the uniform norm.

  \end{lem}
  \begin{proof}
    Let $f_\gamma$ be as in the proof of Lemma \ref{lem:ucurve}.  Define
    $I' \subset I$ to be the subinterval for which
    $\gamma' = F\big( \graph (h_\gamma|_{I'})\big)$, noting that
    $x \mapsto \hat f_\gamma(x) := f_\gamma(x) \modone$ is a $C^2$
    diffeomorphism $I' \to I_{\gamma'}$. Clearly,
  \[
  \rho' = \frac{1}{\int_{I'} \rho \, dx} \cdot \frac{\rho}{|\dot
  f_\gamma|} \circ ( \hat f_\gamma|_{I'})^{-1} \, .
  \]
  For simplicity, assume $\dot f_\gamma > 0$ on $I'$ (either this or
  $\dot f_\gamma < 0$ holds since $\dot f_\gamma \neq 0$ on $I'$);
  otherwise the formulas below differ by a minus sign. We compute
  \begin{align}\label{eq:distortion11}
  \frac{d \log  \rho'}{dx} = \frac{1}{\rho'} \frac{d \rho'}{dx} = \bigg( \frac{1}{\dot f_\gamma} \frac{d \log \rho}{dx} - \frac{\ddot f_\gamma}{(\dot f_\gamma)^2} \bigg) \circ (\hat f_\gamma|_{I'})^{-1} \, ,
  \end{align}
  from which we get the estimate
  \[
  \bigg\| \frac{d \log \rho'}{dx} \bigg\| \leq L^{- \eta} \bigg\|
  \frac{d \log \rho}{dx} \bigg\| + C_0 L^{1 - 2 \eta} \, ,
  \]
  where $C_0 = 8 \pi^2$.
  \end{proof}

  \section{Images of standard pairs and correlation decay}

  Our primary aim in this section is to describe the pushforward
  $F^n_*(\gamma, \rho)$ of a fully-crossing standard pair $(\gamma, \rho)$. In Section 3.1,
  we consider pushing forward measure pairs one timestep, while in Section
  3.2 we will iterate these arguments to describe $F^n_*(\gamma, \rho)$.
  Applications to decay of correlations are derived in Section 3.3.  This
  includes the proof of Theorem~\ref{thm:decayLebesgue}.

  \subsection{Pushing forward standard pairs by $F$}

  Here we describe how to push forward measure pairs of varying
  regularity: fully crossing, standard, and substandard.

  \subsubsection*{Notation and setup}

  For a measure pair $(\gamma, \rho)$, we will describe the pushforward $F_*(\gamma, \rho)$.
  Depending on the regularity (e.g., standard versus substandard) of $(\gamma, \rho)$, we will subdivide
  \[
  F_*(\gamma, \rho) = \mu_{\Lc_{(\gamma, \rho)}} + \mu_{\Ic_{(\gamma, \rho)}} + \mu_{\Jc_{(\gamma, \rho)}} + \mu_{\Ec_{(\gamma, \rho)}} \, ,
  \]
  where
  $\mu_\Lc = \mu_{\Lc_{(\gamma, \rho)}}, \mu_\Ic = \mu_{\Ic_{(\gamma,
  \rho)}}, \mu_\Jc = \mu_{\Jc_{(\gamma, \rho)}}$ are, respectively,
  weighted sums over collections
  $\Lc = \Lc_{(\gamma, \rho)}, \Ic = \Ic_{(\gamma, \rho)}, \Jc =
  \Jc_{(\gamma, \rho)}$ of measure pairs consisting, respectively, of
  fully-crossing standard pairs, standard pairs, and substandard pairs.
  Here, $\mu_\Ec = \mu_{\Ec_{(\gamma, \rho)}}$ is a measure corresponding
  to the portion of $F_*(\gamma, \rho)$ which we do not control (the `error'), and is
  supported on a subset $\Ec = \Ec_{(\gamma, \rho)}$ of $F(\gamma)$ for
  which $F_*(\gamma, \rho)|_\Ec = \mu_\Ec$.

  \begin{nrmk}
  Abusing notation somewhat, when it is clear from context we will use
  $\Lc$ to refer to (i) a collection $\{ (\gamma', \rho')\}$ of
  fully-crossing standard pairs; (ii) a partition of a subset of
  $F(\gamma)$ into fully-crossing u-curves $\gamma'$; and (iii) the
  subset of $F(\gamma)$ itself, i.e., the union over all
  $\gamma' \in \Lc$.  The same applies to each of $\Ic, \Jc$.
  \end{nrmk}


For a measure $\mu$, we write $\| \mu\|$ for the total mass of $\mu$.

\bigskip

We begin by describing the $\Lc, \Ic, \Jc, \Ec$ decomposition when $(\gamma, \rho)$ is a standard pair, not
necessarily fully-crossing.

\begin{lem}\label{lem:cut-standard}
Let $(\gamma, \rho)$ be a standard pair for which $\gamma$ is not necessarily fully-crossing. Then,
\[
F_*(\gamma, \rho) = \mu_\Lc + \mu_\Ec \, ,
\]
where $\| \mu_\Ec\| = O(a_0^{-1} L^{- 1/2})$.
\end{lem}
\begin{proof}
To start, we allocate $F(\gamma \cap \Sc_{1/2})$ to $\Ec = \Ec_{(\gamma, \rho)}$ and subdivide
$\gamma \setminus \Sc_{1/2}$ into at most three connected components $\check \gamma$.

For each $\check \gamma$, in the notation of Lemma \ref{lem:ucurve},
subdivide $\tilde \gamma = \tilde F(\check \gamma)$ into pieces
$\tilde \gamma_n = \tilde \gamma \cap [n, n + 1), n \in \Z$.
Of the nonempty $\tilde \gamma_n$, at most two have length $< 1$;
these are allocated to $\Ec$ , while the $\tilde \gamma_n$ of length $1$
are are projected to $\T^2$ and allocated to $\Lc = \Lc_{(\gamma, \rho)}$.
Distortion is checked as in Lemma \ref{lem:distStdPair} with $\eta = 1/2$; details
are left to the reader.

To estimate $\|\mu_\Ec\|$, we note that $(\gamma, \rho)(\Sc_{1/2}) = O(a_0^{-1} L^{- 1/2})$,
while for any nonempty $\tilde \gamma_n$ as above, we have $(\gamma, \rho)(\tilde F^{-1}(\tilde \gamma_n))
= O(a_0^{-1} L^{- 1/2})$.
\end{proof}

Next, we consider images of substandard pairs.

\begin{lem}\label{lem:cut-sub}
Let $(\gamma, \rho)$ be a substandard pair. Then,
\[
F_*(\gamma, \rho) = \mu_\Ic + \mu_\Jc + \mu_\Ec \, ,
    \]
    where $\|\mu_\Jc\| = O(a_0)$ and $\| \mu_\Ec\| = O(L^{- 1/2})$.
    \end{lem}
    In particular, if $a_0$ is chosen sufficiently small (independently
    of $L$), we have $\| \mu_\Jc\| \leq 1/2$ when $(\gamma, \rho)$ is
    substandard.

    \begin{proof}
    Without loss of generality, let us assume that $\gamma$ has length
    $\in [L^{-1/2}, 2 L^{- 1/2}]$. If not, then subdivide $\gamma$ into
    pieces $\gamma_i$ with lengths $\in [L^{- 1/2}, 2 L^{- 1/2}]$ and
    consider separately each $(\gamma_i, \rho_i)$, where $\rho_{i}$ is the
    renormalized restriction of the density
    {$\rho_i := \big( (\gamma, \rho)(\gamma_i) \big)^{-1}\cdot
    \rho|_{I_{\gamma_i}}$}. Note that by our reduction,
    $\sup_{x_1, x_2 \in I}|\log \frac{\rho(x_2)}{\rho(x_1)}| \le
    \|\frac{d}{dx} \log \rho\| \, |I| = O(L^{\frac12} L^{-\frac12}) =
    O(1)$. Since $\rho$ is a probability density on
    $|I|\approx L^{-\frac12}$, we have $\rho \approx L^{\frac12}$.

    Observe that $\tilde \gamma = \tilde F(\gamma)$ has length
    larger than $L^{1/2} \cdot L^{- 1/2} = 1$. With
    $\tilde \gamma_n = \tilde \gamma \cap ([n, n + 1) \times \T^1)$ as in
    the proof of Lemma \ref{lem:cut-standard}, allocate all fully-crossing
    $\tilde \gamma_n$ to $\Ic = \Ic_{(\gamma, \rho)}$.  At most two
    $\tilde \gamma_n$ remain, each of length $< 1$. For each,
  we distinguish three cases: we add $\tilde \gamma_n$ to
  \begin{enumerate}[(i)]
  \item
  $\Ic$ if $|I_{\tilde \gamma_n}| > a_0$,
  \item $\Jc$ if $|I_{\tilde \gamma_n}| \in [L^{- 1/2}, a_0]$, or
  \item $\Ec$ if $|I_{\tilde \gamma_n}| < L^{- 1/2}$.
\end{enumerate}
In case (ii), note that $\tilde \gamma_n \cap \Sc_{1/2} =
\emptyset$ automatically, since for all
$L$ sufficiently large, the critical strips comprising
$\Sc_{1/2}$ are a distance $> 1/5$ from $\{ x = 0 \} \times
\T^1$, while $I_{\tilde \gamma_n}$ has the form $[n, n + c)$ or $[n +1- c, n
+ 1)$  for some $c \leq a_0 \leq
1/8$.  In order to estimate the contributions to $\mu_\Jc,
\mu_\Ec$, respectively, note that in case (ii) we have $(\gamma,
\rho)(F^{-1} (\tilde \gamma_n)) = O(|I_{F^{-1} (\tilde \gamma_n)}|\, \|\rho\|) = O(a_0 L^{- 1/2}  L^{1/2}) =
O(a_0)$, while in case (iii) we have $(\gamma, \rho)(F^{-1} (\tilde
\gamma_n)) = O(L^{-1/2} \cdot L^{- 1/2}  L^{1/2}) = O(L^{-1/2})$.

It remains to check distortion. For any
$(\gamma', \rho') \in \Ic \cup \Jc$,
  by Lemma \ref{lem:distStdPair} with $\eta =
  1/2$ and the definition of a substandard pair we have
  \[
  \bigg\| \frac{d \log \rho'}{dx} \bigg\| \le L^{- 1/2} \cdot 2 C_0 L^{1/2} + C_0 \leq 3 C_0 \, . \qedhere
  \]
  \end{proof}

  Finally, we consider fully-crossing standard pairs.
  \begin{lem}\label{lem:cut-full}
  Let $(\gamma, \rho)$ be a standard pair for which $\gamma$ is fully-crossing. Then, $F_*(\gamma, \rho)$ admits a representation of the form
  \[
  F_*(\gamma, \rho) = \mu_\Lc + \mu_\Ic + \mu_\Jc + \mu_\Ec \, ,
\]
where $\|\mu_\Ic\| = O(L^{-1/2})$, $\| \mu_\Jc\| = O(a_0 L^{- 1/2})$ and $\| \mu_\Ec\| = O(L^{- 3/4})$.
\end{lem}

\begin{proof}
To start, $F(\gamma \cap \Sc_{1/4})$ is allocated to $\Ec$, giving an $O(L^{- 3/4})$ contribution
to the mass of $\mu_\Ec$.

To allocate $F(\gamma \setminus \Sc_{1/2})$: the set $\gamma \setminus \Sc_{1/2}$ has three
connnected components $\check \gamma$, each of which we handle separately. Fixing a $\check \gamma$
and setting $\tilde \gamma = \tilde F(\check \gamma)$, $\tilde \gamma_n = \tilde \gamma \cap ([n, n + 1) \times \T^1)$,
allocate all $\tilde \gamma_n$ of length $1$ to $\Lc$. For the at-most two remaining nonempty
$\tilde \gamma_n$, allocate to $\Ic, \Jc, \Ec$ according to cases (i) -- (iii) in the proof of Lemma \ref{lem:cut-sub}.
As in Lemma \ref{lem:cut-sub}, in case (ii) we automatically have $\tilde \gamma_n \cap \Sc_{1/2} = \emptyset$. This step contributes
$O(L^{- 1/2})$ mass to $\Ic$; $O(a_0 L^{- 1/2})$-mass to $\Jc$; and $O(L^{-1})$ mass to $\Ec$.
Distortion for measure pairs in $\Lc \cup \Ic \cup \Jc$
allocated so far can be checked using Lemma \ref{lem:distStdPair} with $\eta = 1/2$.

\smallskip

For $F(\gamma \cap (\Sc_{1/2} \setminus \Sc_{1/4}))$, we consider each of the four connected
components $\check \gamma$ of $\gamma \cap (\Sc_{1/2} \setminus \Sc_{1/4})$ separately.
To start, observe that the length of $\tilde \gamma = \tilde F(\check \gamma)$ can be estimated
\[
|I_{\tilde \gamma}| \approx \int_{L^{- 3/4}}^{L^{- 1/2}} L \cdot z \, dz \approx 1 \, .
\]
In particular, $\tilde \gamma \cap \Sc_{1/2}$ has an $O(1)$ number of
connected components.  We allocate each to $\Ec$, contributing
$O(L^{- 1/2} \cdot L^{- 1/4}) = O(L^{- 3/4})$ mass to $\mu_\Ec$.  For
each connected component $\zeta$ of $\tilde \gamma \setminus \Sc_{1/2}$,
allocate $\zeta$ to
\begin{itemize}
	\item[(a)] $\Jc$ if $\zeta$ has length $\in [L^{- 1/2}, a_0]$ or
	\item[(b)] $\Ec$ if $\zeta$ has length $< L^{- 1/2}$.
\end{itemize}
If (c) $\zeta$ has length $> a_0$, then subdivide
$\zeta$ into pieces of length $[a_0/2, a_0]$ and allocate each to $\Jc$.
In cases (a), (c), the contribution to $\mu_\Jc$ is $O(L^{- 1/2})$, while in
case (b) the contribution to $\mu_\Ec$ is $O(L^{- 3/4})$.


To check distortion: for any $(\gamma', \rho') \in \Jc$ with $F^{-1}(\gamma') \subset \gamma \cap (\Sc_{1/2} \setminus \Sc_{1/4})$, we have from Lemma \ref{lem:distStdPair} with $\eta = 1/4$ that
\begin{align*}
\bigg\| \frac{\log d \rho'}{dx}\bigg\| &\le L^{- 1/4} \cdot 3 C_0 + C_0 L^{1/2} \leq 2 C_0 L^{1/2} \, . \qedhere
\end{align*}
\end{proof}

\medskip

\begin{rmk}\label{rmk:parameter}
From this point on, we fix $a_0 \in (0,1/8]$ sufficiently small so
that in Lemmata~\ref{lem:cut-sub} and~\ref{lem:cut-full}, we have
$\|\mu_{\Jc}\| \leq 1/2$. We now treat $a_0$ as a constant parameter
and hereafter omit it from our $O(\cdots)$ estimates.
\end{rmk}

\subsection{Iterated standard pairs}\label{sec:iter}


Fix a fully-crossing standard pair $(\gamma, \rho)$. Below, for each
$n \geq 1$ we define a decomposition
\[
F^n_*(\gamma, \rho) = \mu_{\Lc^n_{(\gamma, \rho)}} +
\mu_{\Ic^n_{(\gamma, \rho)}} + \mu_{\Jc^n_{(\gamma, \rho)}} +
\mu_{\Ec^n_{(\gamma, \rho)}} \, ,
\]
where, as in Section 3.1, each of
$\mu^n_{\Lc} = \mu_{\Lc^n_{(\gamma, \rho)}}, \mu^n_{\Ic} =
\mu_{\Ic^n_{(\gamma, \rho)}}, \mu^n_{\Jc} = \mu_{\Jc^n_{(\gamma,
    \rho)}}$ is a weighted sum of measure pairs of the appropriate
regularity (respectively, fully-crossing, standard, and substandard),
while $\mu_{\Ec^n_{(\gamma, \rho)}} = \mu^n_\Ec$ is a remainder we do
not otherwise control.  We write
$\Lc^n = \Lc^n_{(\gamma, \rho)}, \Ic^n = \Ic^n_{(\gamma, \rho)}, \Jc^n =
\Jc^n_{(\gamma, \rho)}$ for the corresponding classes of, respectively,
fully-crossing, standard and substandard measure pairs, and
$\Ec^n = \Ec^n_{(\gamma, \rho)}$ for the corresponding remainder set.

\subsubsection*{(A) Constructing $\Lc^n, \Ic^n, \Jc^n, \Ec^n$}

To start, we set $\Lc^0 = \{\gamma\}$, $\Ic^0, \Jc^0, \Ec^0 = \{\emptyset\}$. Given
$k \geq 1$, the collections $\Lc^k, \Ic^k, \Jc^k, \Ec^k$, and the measures
$\mu^k_\Lc, \mu^k_\Ic, \mu^k_\Jc, \mu^k_\Ec$,
we define $\Lc^{k+1}, \Ic^{k+1}, \Jc^{k+1}, \Ec^{k+1}$ as follows. Set
\[
\Lc^{k+1} = \bigcup_{(\gamma_k, \rho_k) \in \Lc^k \cup \Ic^k \cup \Jc^k} \Lc_{(\gamma_k, \rho_k)} \, ,
\quad  \quad \mu_\Lc^{k+1}
= \sum_{(\gamma_k, \rho_k) \in \Lc^k \cup \Ic^k \cup \Jc^k} c_{\gamma_k}^k \mu_{\Lc_{(\gamma_k, \rho_k)}} \, ,
\]
where $c^k_{\gamma_k} := F^k_*(\gamma, \rho)(\gamma_k)$. Here, for
measure pairs $(\gamma_k, \rho_k)$, the collections
$\Lc_{(\gamma_k, \rho_k)}$, $\Ic_{(\gamma_k, \rho_k)}$, $\Jc_{(\gamma_k,
  \rho_k)}$ are as in Lemmata~\ref{lem:cut-standard}, \ref{lem:cut-sub},
\ref{lem:cut-full}.  The
$\Ic^{k+1}, \Jc^{k+1}, \mu^{k+1}_\Ic, \mu^{k+1}_\Jc$ are defined
analogously.  Finally, we define
\[
\Ec^{k+1} = F(\Ec^k) \cup \bigcup_{(\gamma_k, \rho_k) \in \Lc^k \cup \Ic^k \cup \Jc^k} \Ec_{(\gamma_k, \rho_k)} \, ,
\quad \quad \mu^{k+1}_\Ec = \sum_{(\gamma_k, \rho_k) \in \Lc^k \cup \Ic^k \cup \Jc^k} c_{\gamma_k}^k \mu_{\Ec_{(\gamma_k, \rho_k)}} \, .
\]
This completes the construction.

\subsubsection*{(B) Estimating mass contributions}

Let us now estimate the relative
sizes of the $\mu^n_\Lc, \mu^n_\Ic, \mu^n_\Jc, \mu^n_\Ec$.

\begin{prop}\label{prop:massEst}
Let $(\gamma, \rho)$ be a fully-crossing standard pair, $n \geq 1$. Then,
\[
F^n_*(\gamma, \rho) =
\mu^n_\Lc +  \mu^n_\Ic +  \mu^n_\Jc + \mu^n_\Ec \, ,
\]
where $\| \mu^n_\Ic \| = \| \mu^n_\Jc \| = O (L^{- 1/2})$ and $\| \mu^n_\Ec\| = O (n L^{- 3/4})$.
\end{prop}

\begin{proof}
From Lemmata~\ref{lem:cut-standard}, \ref{lem:cut-sub}, \ref{lem:cut-full}, we obtain
\begin{align*}
\| \mu_\Ec^{k + 1} \|  &= \| \mu_\Ec^k \| + O\bigg(L^{- 1/2} \cdot \| \mu_\Jc^k \| + L^{- 1/2} \cdot  \| \mu_\Ic^k \| +
L^{- 3/4} \| \mu_\Lc^k \| \bigg) \\
\| \mu_\Jc^{k + 1} \|  &\leq \frac12 \| \mu_\Jc^k \| +  O(L^{- 1/2} \| \mu_\Lc^k \|) \\
\| \mu_\Ic^{k + 1} \|  &= O\big( \| \mu_\Jc^k \|  + L^{- 1/2} \| \mu_\Lc^k \| \big) \\
\| \mu_\Lc^{k + 1} \|  &= \big(1- O (L^{- 1/2}) \big) \| \mu_\Ic^k \| + \big(1- O(L^{- 1/2})\big) \| \mu_\Lc^k \|
\end{align*}
Proposition \ref{prop:massEst} follows by an induction argument, using
the initial state $\|\mu_\Lc^0\| = 1, \| \mu_\Ic^0\| = \| \mu_\Jc^0\| = \| \mu_\Ec^0\| = 0$.
\end{proof}



%
%

If, at time $n$, we discard the curves in $\Ic^n, \Jc^n$, we obtain the following corollary.
\begin{cor}\label{lem:iterateStdPair}
Let $(\gamma, \rho)$ be a fully-crossing standard pair. For any $n \geq 1$, the pushed-forward standard pair $F^n_*(\gamma, \rho)$ admits a representation of the form
\[
F^n_*(\gamma, \rho) = \sum_{(\gamma_n, \rho_n) \in \Lc^n_{(\gamma, \rho)}} c_{\gamma_n} (\gamma_n, \rho_n) + \hat \mu_\Ec^n
\]
where each $(\gamma_n, \rho_n)$ is a fully-crossing standard pair,
 the coefficients $\{c_{\gamma_n} : (\gamma_n, \rho_n) \in \Lc^n_{(\gamma, \rho)}\}$ are nonnegative,
and $\|\hat \mu^n_\Ec\| = O(L^{- 1/2} + n L^{- 3/4})$.
\end{cor}

\subsection{Correlation control for $x$-dependent observables}

We now present some consequences of the arguments in Sections 3.1, 3.2
for correlation decay.  Let $\phi : \T^2 \to \R$ be a $C^1$,
$x$-dependent observable.

\subsubsection{Correlation control for standard pairs}

\begin{prop}[Equidistribution]\label{prop:docStdPair}
Let $(\gamma, \rho)$ be a fully-crossing standard pair
and assume $\int_{\T^1} \phi \, dx= 0$. Then, for all $n \geq 1$ we have that
\begin{align*}
\int \phi \circ F^n \, d (\gamma, \rho)= O\left ( \| \phi \|_{C^0} \cdot \big( (n-1) L^{- \frac34} + L^{- \frac12} \big) \right).
\end{align*}
\end{prop}


First, we prove a preliminary lemma.

\begin{lem}[One-step equidistribution]\label{lem:oneStepDoC}
Let $(\gamma, \rho)$ be a fully crossing standard pair, then
\[
\int \phi \circ F d(\gamma, \rho) = O(\| \phi \|_{C^0} L^{-\frac12}).
\]
\end{lem}

\begin{proof}
We decompose $\gamma \setminus \Sc_{1/2}$ into four pieces according to
membership in the four regions $[0,1/4) \times \T^1, [1/4, 1/2) \times \T^1, [1/2, 3/4) \times \T^1, [3/4, 1) \times \T^1$.
For concreteness, we consider below
the piece $\bar{\gamma} = (\gamma \setminus \Sc_{1/2}) \cap \big( [1/4, 1/2) \times \T^1 \big)$
and will estimate $\int_{\bar \gamma} \phi \circ F \, d (\gamma, \rho)$.
The following considerations can be straightforwardly extended to the other pieces; we leave
this to the reader. Below, we write $\bar \rho : I_{\bar \gamma} \to [0,\infty)$ for the
density for which $(\bar \gamma, \bar \rho)$ is the normalization of
$(\gamma, \rho)|_{\bar \gamma}$.


\medskip

Apply Lemma \ref{lem:cut-standard} to $(\bar \gamma, \bar \rho)$ to obtain the collection
$\Lc = \Lc_{(\bar \gamma, \bar \rho)}$ of fully-crossing standard pairs and
the remainder set $\Ec = \Ec_{(\bar \gamma, \bar \rho)} \subset F(\bar \gamma)$. We have
\[
\int_{\bar \gamma} \phi \circ F \,  d( \bar \gamma, \bar \rho)  =
\int \phi d \mu_{\Lc} + \int \phi d \mu_\Ec =
\sum_{ (\check \gamma, \check \rho) \in \Lc} c_{\check \gamma} \int \phi \, d (\check \gamma, \check \rho)
+ O(\| \phi \|_{C^0} L^{- 1/2}) \, ,
\]
where $c_{\check \gamma} := (\bar \gamma, \bar \rho)(F^{-1} \check \gamma)$.

For each $(\check \gamma, \check \rho) \in \Lc$, we first estimate the $(\check \gamma, \check \rho)$-summand
$\int \phi \, d (\check \gamma, \check \rho) = \int_0^1 \phi \, \check \rho \, d x = \int_0^1 \phi \, (\check \rho - 1) dx$.
Observe from \eqref{eq:distortion11} that
\[
\bigg| \frac{d}{dx} \log \check \rho \bigg| = O \bigg( \frac{1}{| \dot
f_\gamma(x_{\check \gamma})|} + \frac{L}{|\dot f_\gamma(x_{\check
\gamma})|^2} \bigg),
   \]
   where we set $x_{\check \gamma}$ to be the right-endpoint of
   $I_{F^{-1} \check \gamma}$. Checking the simple estimate
   $|\dot f_\gamma (x)| \approx L |x_{\check \gamma} - \frac14|$ on
   $I_{F^{-1} \check \gamma}$, it follows that
   $| \check \rho - 1| = O(L^{-1} |x_{\check \gamma} - \frac14|^{-2})$.
   Putting this all together,
   \[
   \int \phi d (\check \gamma, \rho_{\check \gamma})  = \int_0^1 \phi (\rho_{\check \gamma} - 1) dx = O \left( \| \phi \|_{C^0}
   \cdot L^{-1} \left|x_{\check \gamma} - \frac14\right|^{-2} \right) \, .
   \]

   Let $E_{\check\gamma} = L^{-1}|x_{\check\gamma}-1/4|^{-2}$; since
   $x_{\check\gamma}\in I_{\bar\gamma}$ and
   $I_{\bar \gamma} = [\frac14 + b_L, \frac12]$ where
   $b_L \approx L^{- 1/2}$, we gather that there exists $\bar E = O(1)$ so
   that $E_{\check\gamma} < \bar E$.
   Moreover, note that since
   $c_{\check \gamma} = (\bar \gamma, \bar \rho)(F^{-1} \check \gamma) \approx
   |I_{F^{-1} \check \gamma}|$:
   \begin{align*}
   \sum_{(\check\gamma,\check\rho)\in\cL\ \text{s.t.}\ E_{\check\gamma} >
   z}c_{\check\gamma} = O(L^{-1/2}z^{-1/2})\text{ for any $z > 0$}.
   \end{align*}
   Thus:
   \begin{align*}
   \sum_{ (\check \gamma, \check \rho) \in \Lc} c_{\check \gamma} \int
   \phi \, d (\check \gamma, \check \rho) &\le C\|\phi\|_{C^{0}}\sum_{
   (\check \gamma, \check \rho) \in \Lc} c_{\check\gamma}\,E_{\check\gamma}\\
   &\le C\|\phi\|_{C^{0}} \int_0^{\bar E}dz \sum_{(\check\gamma,\check\rho)\in\cL\ \text{s.t.}\ E_{\check\gamma} > z}c_{\check\gamma} = O(\|\phi\|_{C^{0}}L^{-1/2}).
   \end{align*}
   \ignore{Summing over the $\check \gamma$ and, we conclude
   \[
\bigg| \int \phi \, d \mu_\Lc \bigg|
\leq C \| \phi \|_{C^0} \sum_{\check \gamma} |I_{F^{-1} \check \gamma}| \cdot L^{-1} |x - \frac14|^{-2} \, .
\]
The summation on the RHS is a right-endpoint quadrature of the decreasing function $x \mapsto (x - \frac14)^{-2}$ on the interval $x \in I_{\bar \gamma}$. The quadrature is therefore bounded from above by the corresponding integral over
$I_{\bar \gamma}$. To wit,
\begin{align*}
\bigg| \int \phi \, d \mu_\Lc \bigg| & \leq C \| \phi \|_{C^0} \int_{\frac14 + b_L}^{\frac12} L^{-1} (x - \frac14)^{-2} dx \, , \\
\end{align*}
where $I_{\bar \gamma} = [\frac14 + b_L, \frac12]$. Noting that $b_L \approx L^{- 1/2}$, we compute
\[
\bigg| \int \phi \, d \mu_\Lc \bigg|   
\leq C \| \phi \|_{C^0} L^{-1}  \bigg( - (t - 1/4) ^{-1} \bigg|_{\frac14 + b_L}^{1/2}\bigg)
\leq C  \| \phi \|_{C^0} (L b_L)^{-1} = O(\| \phi \|_{C^0} L^{- 1/2}) \,
\]}
where $C > 0$ is an absolute constant independent of $L$ and $\phi$. This completes the proof. \qedhere

\end{proof}

\begin{proof}[Proof of Proposition \ref{prop:docStdPair}]
Apply Corollary \ref{lem:iterateStdPair} to $F^{n-1}_*(\gamma, \rho)$ to obtain
\[
F^{n-1}_* (\gamma, \rho) = \sum_{(\gamma_{n-1}, \rho_{n-1}) \in \Lc^{n-1}_{(\gamma, \rho)} }
c_{\gamma_{n-1} } (\gamma_{n-1}, \rho_{n-1}) + \hat \mu^{n-1}_\Ec \, .
\]
Then,
\begin{align*}
\int \phi \circ F^n \, d (\gamma,\rho) & = \sum_{(\gamma_{n-1}, \rho_{n-1} ) \in \Lc^{n-1}_{(\gamma, \rho)}} c_{\gamma_{n-1}} \int \phi \circ F \, d (\gamma_{n-1}, \rho_{n-1}) \\
& + O( \| \phi \|_{C^0} \cdot ((n-1) L^{- 3/4} + L^{- 1/2}) \, .
\end{align*}
The proof is complete on applying Lemma \ref{lem:oneStepDoC} to each summand.
\end{proof}

\subsubsection{Correlation control for Lebesgue measure}

Using the equidistribytion estimate for standard pairs and the machinery
developed so far, we conclude this section with the proof of
Theorem \ref{thm:decayLebesgue}, reformulated below as Corollary~\ref{cor:LebDoC}.
\begin{cor}\label{cor:LebDoC}%
Let $\phi, \psi : \T^1 \to \R$ be $C^1$, $x$-dependent
observables. Then,
  \begin{align*}%
  \int \psi \cdot \phi \circ F^n \, d \Leb - \int \psi \int
  \phi = O\left(\| \psi \|_{C^1} \| \phi\|_{C^1} \cdot \big( (n -1) L^{-\frac34} +
  L^{-\frac12} \big) \right)
  \end{align*}
  \end{cor}
  \begin{proof}
  Let $c > 0$ be a constant, to be specified later, and define
  $\hat \psi = \hat \psi_c = \frac{\psi + c}{\int \psi + c}$. Define
  $\hat \phi = \phi - \int \phi$. For each fixed $y \in \T^1$, we intend to
  apply Proposition \ref{prop:docStdPair} to
  \begin{align*}
  \int \hat \phi \circ F^n \, d (\gamma^y, \hat \psi) \, ,
  \end{align*}
  where $\gamma^y := \T^1 \times \{ y \}$
  and we regard $\hat \psi$ as a density function on $\T^1$ as in the
  definition of a standard pair. To make this legitimate, the parameter
  $c > 0$ must be chosen so $\hat \psi$ is (i) nonnegative and (ii)
  satisfies the distortion estimate \eqref{eq:dist}.  For this,
  \[
  \frac{d \log \hat \psi}{dx} = \frac{1}{\hat \psi} \frac{d \hat \psi}{dx} = \frac{1}{\psi + c} \frac{d \psi}{dx} \, ,
  \]
  hence
  $| \frac{d \log \hat \psi}{dx} | \leq \frac{1}{c - \| \psi \|_{C^1}}
  \| \psi\|_{C^1}$
  Taking $c = 2 \| \psi \|_{C^1}$ yields (i) $\psi + c > 0$ and
  (ii) $| \frac{d \log \hat \psi}{dx} | \leq 1 \leq 3 C_0$, as needed.

  Applying Proposition \ref{prop:docStdPair} for each fixed $y$, then
  integrating over $y \in \T^1$, we have
  \begin{align*}
  \int \hat \psi \cdot \hat \phi \circ F^n \, d \Leb & =
  \int \bigg( \int \hat \phi \circ F^n \, d (\gamma^y, \hat \psi) \bigg) dy \\
  & = O \bigg( \| \phi \|_{C^1} \big( (n-1) L^{-\frac34} + L^{-\frac12}  \big) \bigg) \, ,
  \end{align*}
  while
  \begin{align*}
  \int \hat \psi \cdot \hat \phi \circ F^n \, d \Leb & = \frac{1}{\int \psi + c} \bigg( \int (\psi + c) \cdot \phi \circ F^n \, d \Leb - \int \phi \int (\psi + c) \bigg) \\
  & = \frac{1}{\int \psi + c} \bigg( \int \psi  \cdot \phi \circ F^n \, d \Leb - \int \phi \int \psi  \bigg)
  \end{align*}
  holds since $F$ preserves $\Leb$. This completes the proof.
  \end{proof}
\section{Central Limit Theorem}

Let $\phi$ be a $C^1$, $x$-dependent observable.
We obtain in this section a Central Limit Theorem (CLT) for sequences of the form
\begin{align}\label{e_birkhoffSum}
S_{N, L}\phi := \sum_{i = 0}^{N-1} \phi \circ F^i_L
\end{align}

where $N = N(L)$ is a suitably chosen function of $L \geq 0$ which
increases sufficiently slowly as in the assumptions of
Theorem~\ref{thm:CLT}.

\medskip

For this, we follow the standard route of obtaining a martingale
difference approximation for the sequence $S_{N, L} \phi$. The plan is
as follows. In Section 3.1 we will define, for each $L$, a filtration of
$\T^2$ by $F_L$-preimages of fully-crossing standard pairs (plus a small
remainder which we do not control). In Section 3.2, we will define a
martingale difference approximation $\tilde S_{N, L} \phi$, and show how
a CLT for the approximation implies a CLT for the original
$S_{N, L} \phi$.  Finally, in Section 3.3 we apply a result of McLeish
(see \cite{McL1974}) on CLTs for martingale difference arrays to conclude
the CLT for $\tilde S_{N, L} \phi$, thereby completing the proof of
Theorem \ref{thm:CLT}.

\medskip

{\it Notation for Section 4. } Since this section has more of a
probabilistic flavor, we will at times write $\P$ for Lebesgue measure
on $\T^2$ and $\E$ for the expectation with respect to $\P$. Given a
$\sigma$-algebra $\Fc \subset \operatorname{Bor}(\T^2)$, we write
$\E(\cdot | \Fc)$ for the conditional expectation w.r.t. $\Fc$.

At times in Section 4, when $L$ is fixed or when clear from context, we
will write $F = F_L$.



\subsection{Filtration by u-curves}
Fix $L > 0$ sufficiently large for the purposes of the results in
Section 3.  By the end of Section 4.1, we will have constructed a
sequence of $\sigma$-algebras $\Uc_i = \Uc_i(L), i \geq 1$, each
generated by a partition of $\T^2$ into fully crossing curves, plus some
small remainder set, with the property that
$F_L \Uc_i \subset \Uc_{i + 1}$. As a result, the pull-backs
$\Fc_i = \Fc_i(L) := F^{-i}_L \Uc_i$ comprise a filtration on
$\T^2$. This is the filtration we will use to define our martingale
approximation in Section 4.2.


\medskip

{\it Notation. }For $y \in \T^1$, let $\gamma^y := \T^1 \times \{ y \}$,
which is clearly a fully-crossing u-curve.  Writing
${\bf 1} : (0,1) \to \R$ for the density identically equal to $1$, we
regard $(\gamma^y, {\bf 1})$ as a fully-crossing standard pair. Applying
the machinery in Section \ref{sec:iter}, for $n \geq 0$ we define the
collections of measure pairs
\[
\Lc^n_y = \Lc^n_{(\gamma^y, {\bf 1})} \, , \quad \Ic^n_y =
\Ic^n_{(\gamma^y, {\bf 1})} \, , \quad \Jc^n_y = \Jc^n_{(\gamma^y,
{\bf 1})} \, ,
\]
and the remainder set $\Ec^n_y = \Ec^n_{(\gamma^y, {\bf 1})}$. Define
the partition $\Pc^n_y$ of $F^n(\gamma^y)$ by
\[
\Pc^n_y = \Lc^n_y \cup \Ic^n_y \cup \Jc^n_y \cup \{ \Ec^n_y\} \, ,
\]
where the $\Lc^n_y, \Ic^n_y, \Jc^n_y$ are treated above as collections
of u-curves, and $\{ \Ec^n_y\}$ is the trivial partition on $\Ec^n_y$.

\smallskip

Below, for partitions $\a, \b$ on the same space, we write $\a \leq \b$
if each $\a$-atom is a union of $\b$-atoms (\ie $\alpha$ is coarser than
$\beta$).  We write $\a \vee \b$ for the {\emph join} of $\a$ and $\b$,
i.e., the partition of the form $\{ C \cap D : C \in \a, D \in \b\}$.
Clearly, if $\a\le\b$, then $\a\vee\b = \b$. Given a partition $\a$ we
denote with $\sigma(\a)$ the $\sigma$-algebra generated by $\a$; notice
that if $\a\le\b$ we have $\sigma(\a)\subset\sigma(\b)$.



\medskip

\subsubsection*{Construction of $\Uc_i$}  We are about to construct inductively a
sequence $\Xi_i = \Xi_i(L)$ of measurable partitions of $\T^2$ into
(mostly) fully crossing curves with the property that
$F \Xi_i \leq \Xi_{i +1}$.  The $\sigma$-algebras $\Uc_i$ will be of the
form $\Uc_i = \sigma( \Xi_i)$, and the property
$F \Uc_i \subset \Uc_{i + 1}$ will follow by the remark made above.

We set $\Xi_0$ to be the partition of $\T^2$ into the u-curves
$\{\gamma^y\}_{y \in \T^1}$.  Assume by induction that we have defined
the partitions $\Xi_0, \cdots, \Xi_i$, we will construct $\Xi_{i + 1}$
on $\T^2$ as follows: we define $\Xi_{i + 1}|_{F^{i+1}(\gamma^y)}$
separately for each $y \in \T^1$.  For fixed  $y$, we set
\begin{align*}
  \Xi_{i + 1}|_{F^{i+1}(\gamma^y)} = F(\Xi_i|_{F^i(\gamma^y)})
  \vee \Pc^{i+1}_y \, .
\end{align*}
Reconstituting $\Xi_{i + 1}$ from its definition on each atom of
$F^{i+1}(\Xi_0) = \{ F^{i+1}(\gamma^y)\}_{y \in \T^1}$,
it is clear that $F(\Xi_i) \leq \Xi_{i+1}$, as desired.

\smallskip

Having constructed the $\Uc_i$, we define the sequence of $\sigma$-algebras
\begin{align*}
  \Fc_i = F^{-i} \Uc_i \, , \quad i \geq 1 \, ,
\end{align*}
  which is clearly seen to be
an increasing filtration on $\T^2$. Moreover, the partition
$\Pc_y^n$ depends measurably on $y$, (in fact, on a piecewise
continuous fashion);  from this
it is not hard to check that each of the $\Uc_i, \Fc_i, i \geq 1$ is contained in $\Bor(\T^2)$.

\subsubsection*{Properties of the $\Uc_i$ } Let us record some basic facts for
future use.  Set $\tilde G^n = \tilde G^n(L) = \bigcup_y \Lc^n_y$, where
$\Lc^n_y$ is regarded as a subset of $F^n(\gamma^y)$.  Then,
$\Gamma^n = \Gamma^n(L) := \Xi_n|_{\tilde G^n}$ is a partition of
$\tilde G^{n}$ consisting of fully-crossing u-curves, coinciding with
the union $\cup_y \Lc^n_y$ of u-curves. We continue to abuse notation
and write $\Gamma^n$ for both the collection of u-curves and the
corresponding collection of standard pairs $\cup_y \Lc^n_y$.  We set
$G^n = G^n(L) := F_L^{-n} \tilde G^n$ and
$B^n = B^n(L) : = \T^2 \setminus G^n$.

\begin{lem}\label{lem:filtration}
For each $n \geq 1$, the following holds.
\begin{itemize}
	\item[(a)] We have $\Leb B^n = O((n-1) L^{-3/4} + L^{- 1/2})$.
	\item[(b)] Restricted to the set $F^{-1}\tilde{G}^n$, the $\sigma$-algebra $F^{-1} \Uc_n$ is
	generated by atoms of $F^{-1} \Gamma^n$, each of which has diameter bounded
	from above by $L^{- \frac12}$.
\end{itemize}
\end{lem}


In the coming proofs, we routinely take conditional expectations with
respect to the $\sigma$-algebras $\{\Uc_n\}$.  Below we record how these
computations are carried out.


\begin{lem}\label{lem:representCondExpectation}
  Let $\psi : \T^2 \to \R$ be a $C^0$ function. Then, there is a version
  of the conditional expectation $\E(\psi | \Uc_n)$ of $\psi$ with
  respect to $\Uc_n$ with the property that for every
  $(\gamma_n, \rho_n) \in \Gamma^n$, we have
  \begin{align*}
    \E(\psi | \Uc_n) = \int \psi \, d(\gamma_n, \rho_n)
    &= \int \psi(x, h_{\gamma_n}(x)) \, \rho_{n}(x) \, dx \quad \text{ on } \gamma_n \, .
  \end{align*}
\end{lem}
Hereafter we intentionally confuse $\E (\psi | \Uc_n)$ with the
expression on the right-hand side.
\begin{lem}\label{l_conditionalEstimate}
  Let $\phi : \T \to \R$ be a $C^0$ function with zero average. Then we
  have, for any $0\le i < L^{1/4}$:
  \begin{align*}
   \E|\E(\phi | \Uc_i)| = O(\|\phi\|_{C^{0}}L^{-1/2})
  \end{align*}
\end{lem}
\begin{proof}
  First observe that since $\Uc_0$ is the trivial $\sigma$-algebra (on
  each horizontal curve) we have $\E(\phi|\Uc_0) = \int_0^{1}\phi = 0$
  by assumption.  Hence, we can assume $i \ge 1$.  Let us denote by
  $B_{i-1, i}$ the union of $F(\Jc_{i-1})$, $\Ic_i$, $\Jc_i$, $\Ec_i$;
  here we use the shorthand $\Ic_i := \cup_y \Ic^i_y$, with
  $\Jc_i, \Ec_i$ defined analogously.

  According to Proposition~\ref{prop:massEst}, the set $B_{i-1, i}$ has Lebesgue
  measure $O(L^{-1/2})$; moreover, the complement of $B_{i-1,i}$ is a
  union of fully crossing curves $\gamma$ such that $F^{-1}(\gamma)$ is
  contained in either $\Ic_{i-1}$ or $\Lc_{i-1}$. Let $\rho$ be the
  density supported on $\gamma$, and $(\gamma', \rho')$ be such that
  $F_*(\gamma', \rho') = (\gamma, \rho)$. Since $(\gamma', \rho')$ is
  contained in a standard pair, we have
  $\|\frac{d}{dx}\log \rho'\| = O(1)$.  Moreover, due to the way $\Lc_i$
  is constructed (Lemma \ref{lem:cut-standard}, \ref{lem:cut-full}),
  $\gamma'$ is disjoint from the critical set $\Sc_{1/2}$. Then
  by Lemma~\ref{lem:distStdPair},
  $\|\frac{d}{dx}\log \rho\| = O(L^{-1/2})$ and
  $\rho = 1 + O(L^{-1/2})$. We conclude that on almost every point
  $(x, y) \in \T^2 \setminus B_{i-1, i}$, there exists a fully crossing
  standard pair $(\gamma, \rho)$ such that
  \[
    \E(\phi | \Uc_i)(x,y) =  \int_0^1 \phi(x') \rho(x') \, dx' = O(\| \phi \|_{C^0} L^{- 1/2}).
  \]
  Combined with the measure estimate for $B_{i-1, i}$, we conclude
  $\bE|\E(\phi | \Uc_i)| = O(\|\phi\|_{C^0}L^{-1/2})$.
\end{proof}
Lastly, for observables $\psi : \T^2 \to \R$ we recall the identities
\begin{align*}
\E(\psi\circ F_L | \Uc_n(L)) &= \E(\psi | F_L \Uc_n(L))\circ F_L\\
\E(\psi| \Uc_n(L))\circ F_L &= \E(\psi \circ F_L| F_L^{-1} \Uc_n(L)).
\end{align*}
which follow from the definition and will be used several times in the
sequel.
\subsection{Martingale difference approximation}


From this point on, an increasing function $N : \R_{>0} \to \N$ is fixed
for which the condition
\begin{equation}
\label{eq:N14}
N(L) \cdot L^{- \frac{1}{4}} \to 0 \quad \quad \text{ as } L \to \infty \, ,
\end{equation}
as in the hypotheses of Theorem \ref{thm:CLT}, is assumed to hold.  We
let $\phi : \T^1 \to \R$ be a $C^1$ observable with $\int \phi dx = 0$
and assume $\phi$ is not identically zero; in particular we have
$\int \phi^2 dx > 0$.

\medskip

We intend to approximate the Birkhoff sum $S_{N, L} \phi$ (defined
in~\eqref{e_birkhoffSum}) by $\tilde S_{N, L} \phi$, which we define
as:
\begin{align*}
\tilde S_{N, L} \phi &= \sum_{i = 1}^N \E(\phi \circ F_L^{i-1} | \Fc_i(L))
= \sum_{i = 1}^N \E(\phi | F_L^{-1} \Uc_{i}(L)) \circ F^{i-1}_L.
\end{align*}
\begin{lem}\label{lem:approxInProbability}
Under condition \eqref{eq:N14}, we have
$\frac{1}{\sqrt{N(L)}} |S_{N(L), L} \phi - \tilde S_{N(L), L} \phi |
\to 0$ in probability with respect to Lebesgue measure.
\end{lem}
In particular, the convergence in distribution of
$\frac{1}{\sqrt{N(L)}} \tilde S_{N(L), L} \phi $ to a centered Gaussian
$\Nc(0, \sigma^2)$ is equivalent to the convergence in distribution of
$\frac{1}{\sqrt{N(L)}}S_{N(L), L} \phi$ to the same law $\Nc(0, \sigma^2)$.
\begin{proof}
For the sake of readability, in the following proof we drop the “$L$”
and write $S_N = S_{N(L), L} \phi$,
$\tilde S_N = \tilde S_{N(L), L} \phi$, $\Uc_i = \Uc_i(L)$, $N = N(L)$
and $F = F_L$.

\medskip

We start by examining the $i$-th summand of $\tilde S_N$, \ie $\E(\phi
\circ F^{-1} | \Uc_i)$.  If we evaluate the conditional expectation on
some point of $\tilde G_n$, Lemma~\ref{lem:representCondExpectation} provides
\[
\E(\phi \circ F^{-1} | \Uc_i) = \int \phi \circ F^{-1}(x, h_{\gamma_i}(x)) \, \rho_{i}(x) dx
\]
when the left-hand side is evaluated on the (fully crossing) standard
pair $(\gamma_i, \rho_i) \in \Gamma^i$.  Fixing $(\gamma_i, \rho_i)$,
let $(\gamma_{i-1}, \rho_{i-1}) \in \Gamma^{i-1}$ be such that
$\gamma_i \subset F(\gamma_{i-1})$.  Observe that $f_{\gamma_{i-1}}$
maps some interval $\tilde I_{\gamma_i}$ diffeomorphically onto
$[0,1)$. By the change of variables formula,
\begin{align*}
\int \phi \circ F^{-1}(x, h_{\gamma_i}(x)) \, \rho_{i}(x) dx =
\frac{1}{\int_{\tilde I_{\gamma_{i}}} \rho_{{i-1}} dx} \int_{\tilde
I_{\gamma_{i}}} \phi(x) \, \rho_{i-1}(x) dx
\end{align*}
By Lemma~\ref{lem:filtration}(b), the length of $\tilde I_{\gamma_{i}}$
is $\leq L^{- 1/2}$, and so for $(x, y) \in F^{-1}(\gamma_i)$ the right
hand side above equals $\phi(x) + O(\| \phi \|_{C^1} L^{- 1/2})$. Thus
\begin{align}\label{eq:smallAtomsAverage}
\E(\phi | F^{-1} \Uc_i) = \phi(x) + O(\| \phi \|_{C^1} L^{- \frac12}) \quad \text{ on } F^{-1}\tilde G^i \, .
\end{align}
We conclude that
$\frac{|\tilde S_N - S_N|}{\sqrt{N}} \leq \sqrt N L^{- 1/2} \| \phi
\|_{C^1}$ holds on $\bigcap_{n = 1}^N G^n$. By \eqref{eq:N14}, the
quantity on the right hand side goes to $0$ as $L \to \infty$.

To complete the proof of convergence in probability, it suffices to show
that $\P(\bigcup_{n = 1}^N B^n)$ converges to $0$ as $L \to \infty$. For this, from
the estimate in Lemma \ref{lem:filtration} (a) we have
$\P(\bigcup_{n=1}^N B^n) = O(N^2 L^{- 3/4} + N L^{- 1/2})$, which also goes
to $0$ as $L \to \infty$ under \eqref{eq:N14}.
\end{proof}

\subsubsection{Representation of $\tilde S_{N, L}$ as a sum of
martingale differences}

In the next lemma, we represent $\tilde S_N$ as a sum of the form
$\tilde S_N = \sum_{i = 1}^N U_i$, where the $U_i = U_i(L)$ are
martingale differences with respect to the filtration
$(\Fc_i(L))_i$. Below, we use the convention
$\Fc_0 = \{ \T^2, \emptyset\}$.
\begin{lem}
Fix $L$ and define
\[
U_i = \sum_{m = i}^N \bigg( \E(\phi \circ F^{m-1} | \Fc_i) - \E(\phi \circ F^{m-1} | \Fc_{i - 1}) \bigg)  \, .
\]
\begin{itemize}
	\item[(a)] The sequence $(U_i)_{i = 1}^N$ is a martingale difference, i.e., each $U_i$ is $\Fc_i$-measurable and
	$\E(U_i | \Fc_{i-1}) = 0$ for all $1 \leq i \leq N$; and
	\item[(b)] we have $\tilde S_{N, L} = \sum_{i = 1}^N U_i$.
\end{itemize}
\end{lem}
\begin{proof}
Item (a) is obvious. For (b), we compute:
\begin{align*}
\sum_{i = 1}^N U_i & = \sum_{i = 1}^N \sum_{m = i}^N \bigg( \E(\phi \circ F^{m-1} | \Fc_i) - \E(\phi \circ F^{m-1} | \Fc_{i-1}) \bigg) \\
& = \underbrace{\sum_{i = 1}^N \E(\phi \circ F^{i-1} | \Fc_i)}_{= \tilde S_N} + \underbrace{\sum_{i = 1}^N  \sum_{m = i + 1}^N \E(\phi \circ F^{m-1} | \Fc_i)}_{I}- \underbrace{ \sum_{i = 1}^N \sum_{m = i}^N \E(\phi \circ F^{m-1} | \Fc_{i-1}) }_{II}
\end{align*}

For the $I$ term, the $i = N$ summand is empty, and so
\[
I = \sum_{i = 1}^{N-1} \sum_{m = i + 1}^N \E(\phi \circ F^{m-1} | \Fc_i)
\]
For the $II$ term, the $i = 1$ summand is zero since $\Fc_0$ is the trivial $\sigma$-algebra. On replacing $i \mapsto i + 1$,
\[
II = \sum_{i = 2}^N \sum_{m = i}^N \E(\phi \circ F^{m-1} | \Fc_{i-1}) = \sum_{i = 1}^{N-1} \sum_{m = i + 1}^N \E(\phi \circ F^{m-1} | \Fc_i)
\]
and so $I = II$. We conclude $\sum_{i = 1}^N U_i = \tilde S_N$.
\end{proof}
\subsubsection{Asymptotic estimate for $U_i$}

Before continuing, we give the following asymptotic estimate on the $U_i$.

\begin{prop}\label{prop:asymptUi}
For each $1 \le i \le N$, the function
\[
V_i = U_i -  \phi \circ F^{i-1}
\]
satisfies $V_i = O(N \| \phi\|_{C^0})$ and
$\bE|V_i| = O( \| \phi \|_{C^1} N L^{-\frac12})$.
\end{prop}
\begin{proof}
We expand
\begin{align*}
V_i & =  \underbrace{\E(\phi | F^{-1} \Uc_i ) \circ F^{i-1} - \phi \circ F^{i-1}}_{(a)}
+ \underbrace{\E(\phi | \Uc_i) \circ F^i- \E(\phi | \Uc_{i-1}) \circ F^{i-1}}_{(b)} \\
& + \underbrace{\sum_{j = 1}^{N-i-1} \E(\phi \circ F^j | \Uc_i) \circ F^i - \sum_{j = 1}^{N-i} \E(\phi \circ F^j | \Uc_{i-1}) \circ F^{i-1}}_{(c)}
\end{align*}
Clearly $|V_i| = |U_i - \phi\circ F^{i-1}| = O(N \|\phi\|_{C^0})$, and
so we are left only to show the second bound.

In the estimates below, we make liberal use of the fact that under
\eqref{eq:N14}, we have $N L^{- 3/4} = o(L^{- 1/2})$, hence the term
$O(N L^{- 3/4} + L^{- 1/2})$ appearing in the error estimate for
Proposition \ref{prop:docStdPair} can be written $O(L^{- 1/2})$.

\noindent\textbf{Term (a):}
From \eqref{eq:smallAtomsAverage},
\begin{align*}
  | \phi \circ F^{i-1} - \E(\phi | F^{-1}\Fc_i)\circ F^{i-1}  | = O(\| \phi\|_{C^1} L^{- 1/2}) \text{ holds on } G^i.
\end{align*}
The component on $B^i$ has expectation
$O(\| \phi \|_{C^0} L^{-\frac12})$, since $\leb(B^i) = O(L^{-\ifrac12})$
by Lemma~\ref{lem:filtration}(a).  In total,
$\E | (a)| = O ( \| \phi \|_{C^1} L^{- 1/2}) $.

\bigskip
\ignore{\noindent {\bf Term (b): } 
Evaluating at $(\gamma_i, \rho_i) \in \Gamma^i$, we have
\notej{I do not understand the following estimate: why is it small?
  It seems it should be the same estimate as in (c), but it does not work
since there is no composition with $F$. What am I not understanding?}
\[
  \E(\phi | \Uc_i) =  \int_0^1 \phi \cdot \rho_i \, dx = O(\| \phi \|_{C^0} L^{- 1/2})
\]
by Proposition~\ref{prop:docStdPair}. Similarly, when evaluated at
$(\gamma_{i-1},\rho_{i-1}) \in \Gamma^{i-1}$, we have
\[
\E(\phi | \Uc_{i-1}) = O(\| \phi \|_{C^0} L^{ -1/2} )  \, .
\]
Since $\P(B^i), \P(B^{i-1})$ are each $O( L^{-\frac12})$, we obtain
$\E |(b)| = O( | \phi \|_{C^0} L^{- 1/2})$.
\bigskip
}
\noindent{\bf Term (b)}: by Lemma~\ref{l_conditionalEstimate} we
conclude
$\bE \big( |\E(\phi | \Uc_i)| \circ F^i \big) = \bE|\E(\phi | \Uc_i)| =
O(\|\phi\|_{C^0}L^{-1/2})$. The term
$\E(\phi | \Uc_{i-1}) \circ F^{i-1}$ of course satisfies identical estimates.


\bigskip

\noindent {\bf Term (c):} Evaluating at $(\gamma_i, \rho_i) \in \Gamma^i$, we have
\[
\E( \phi \circ F^j | \Uc_i)  = \int \phi \circ F^j \, d (\gamma_i, \rho_{i}) = O\big(\| \phi \|_{C^0}  L^{-\frac12} \big)
\]
by Lemma \ref{lem:representCondExpectation} and Proposition \ref{prop:docStdPair}. Similarly, when evaluated at $\gamma_{i-1} \in \Gamma^{i-1}$,
\[
\E(\phi \circ F^j | \Uc_{i-1}) = O\big( \| \phi \|_{C^0}  L^{-\frac12} \big) \, .
\]
The expectations on the bad sets $B^i, B^{i-1}$ are again $O(\| \phi \|_{C^0} L^{-\frac12})$. Since there are at most $N$ such terms, we have $\E|(c)| = O(\| \phi \|_{C^0} N L^{-\frac12})$. Summing (a), (b), (c) completes the proof.
\end{proof}

\begin{cor}\label{cor:U2}
For $1 \le i \le N$, the function
\[
W_i = U_i^2 - \phi^2 \circ F^{i-1}
\]
satisfies $W_i = O(N^2 \| \phi \|_{C^0}^2)$ and $\bE|W_i| = O(\| \phi \|_{C^1}^2 N^2 L^{-\frac12})$.
\end{cor}
\begin{proof}
  The estimate $W_i = O(N^2\|\phi\|_{C^{0}})$ is straightforward and
  left to the reader.  In order to estimate $\E|W_i|$, observe that
  \begin{align*}
    W_i =  2 (\phi \circ F^{i-1}) V_i + V_i^2 \, .
  \end{align*}
  Then, from Proposition~\ref{prop:asymptUi} we estimate
  $\bE( 2 (\phi \circ F^{i-1}) V_i) \le 2\|\phi\|_{C^0} \bE(|V_i|) =
  O(\|\phi\|_{C^1}^2 N L^{-\frac12})$, and
  $\bE(V_i^2) \le \sup(|V_i|) \bE(|V_i|) = O( \| \phi \|_{C^1}^2 N^2
  L^{-\frac12})$.
\end{proof}

\subsection{Central Limit Theorem for the martingale approximation}
Lemma \ref{lem:approxInProbability} reduces Theorem \ref{thm:CLT} to
verifying the same Central Limit Theorem for
$\frac{1}{\sqrt{N(L)}} \tilde S_{N(L), L} \phi $ as $L \to \infty$.  We will
obtain this using the following result due to McLeish.

\begin{thm}[\cite{McL1974}]\label{thm:mcl}
  Let $(\Omega, \Fc, \P)$ be a probability space. Let
  $\{ k_n\}_{n \geq 1}$, be an increasing sequence of whole numbers
  tending to infinity, and for each $n \geq 1$, let
  $\Fc_{1, n} \subset \Fc_{2, n} \subset \cdots \subset \Fc_{k_n, n}
  \subset \Fc$ be an increasing sequence of sub-$\sigma$ algebras of
  $\Fc$.   For each such $n, i$, let $X_{i, n}$ be a random variable,
  measurable with respect to $\Fc_{i, n}$, for which
  $\E(X_{i, n} | \Fc_{i-1, n}) = 0$, and write
  $Z_n = \sum_{1 \leq i \leq k_n} X_{i, n}$. Assume
\begin{itemize}
	\item[(M1)] $\max_{i \leq k_n} |X_{i, n}|$ is uniformly bounded, in $n$, in the $L^2$ norm;
	\item[(M2)] $\max_{i \leq k_n} |X_{i, n}| \to 0$ in probability as $n \to \infty$; and
	\item[(M3)] $\sum_{i = 1}^{k_n} X_{i, n}^2 \to 1$ in probability as
      $n \to \infty$.
\end{itemize}
Then, $Z_n$ converges weakly to a standard Gaussian.
\end{thm}

Given an arbitrary increasing sequence $L_n \to \infty$, we intend to apply this theorem to the array
\begin{align}\label{e_arrayDefinition}
  X_{i, n} &:= \frac{U_i(L_n)}{ \sqrt{\sum_{i = 1}^{k_{n}} \E \big( U_i(L_n) \big)^2} }\, , &\Fc_{i, n} &= \Fc_i(L_n) \, , &k_n &= N(L_n) \, .
\end{align}
Assuming this can be done, we will have proved that
\[
  \frac{\tilde S_{N(L_n), L_n} \phi}{ \sqrt{\sum_{i = 1}^{k_{n}}%
      \E \big( U_i(L_n) \big)^2} }
\]
converges to a standard Gaussian $\Nc(0,1)$. Afterwards,
Theorem~\ref{thm:CLT} easily follows from the asymptotic estimate for
$\sum_i \E \big(U_i(L_n) \big)^2$ given below.
\begin{prop}\label{prop:asympVariance}
Under condition \eqref{eq:N14}, we have for all $L$ sufficiently large that
\[
\sum_{i = 1}^{N(L)} \E \big(U_i(L) \big)^2 = N(L) \int \phi^2 + o(N(L)) \, .
\]
\end{prop}
\begin{proof}
  Dropping the “$L$” and using Corollary~\ref{cor:U2}, we estimate
  \begin{align*}
  \sum_{i = 1}^N \bE(U_i^2) 
    &= \sum_{i = 0}^{N}\int \phi^2 \circ F^{i-1} + \sum_{i=1}^N \bE(W_i) \\
    &= N \int \phi^2 + O(\| \phi \|_{C^1}^2 N^3 L^{-\frac12}) = N
  \int \phi^2 + o( \| \phi \|_{C^1}^2 N) \,.  \qedhere
  \end{align*}
\end{proof}

 It remains to verify the hypotheses (M1) -- (M3) in Theorem
\ref{thm:mcl} for our choice of $X_{i, n}$.  In the following estimates,
we write $L = L_n$ and otherwise drop the “$L$” from our notation
whenever possible.  Moreover, to improve readability we will drop
$\| \phi \|_{C^1}$ terms from our estimates, absorbing them into the
$O(\cdots), o(\cdots)$ notation.

\subsubsection*{Proof of (M1) and (M2) in Theorem \ref{thm:mcl}}\ \\
In fact we will prove
\begin{align*}
  \int \max_{i \leq k_n} X_{i, n}^2  \to 0 \quad \text{ as }  n \to \infty \, ,
\end{align*}
which implies both (M1) and (M2).
Using~\eqref{e_arrayDefinition}, Proposition~\ref{prop:asympVariance},
Corollary~\ref{cor:U2} we estimate
\begin{align*}
  \int \max_{i\leq k_n} X_{i,n}^2
  &= \int \frac{\max_{i \leq N} U_i^2}{N (\int \phi^2 + o(1))} \le \int
    \frac{\max_{i\le N} \phi^2 \circ F^{i-1} +  \max_{i\le N} |W_i|}{N (\int \phi^2 + o(1))} \\
  & \le \frac{O(1) + \sum_{i = 1}^{N}\bE|W_i|}{N (\int \phi^2 + o(1))}
    = \frac{O(1)+ O(N^3 L^{-\frac12})}{N (\int \phi^2 + o(1))} \, ,
\end{align*}
which, under \eqref{eq:N14}, goes to $0$ as $L = L_n \to \infty$.

\subsubsection*{Proof of (M3)}
We write
\begin{align*}
  \sum_i X_i^2 - 1 &= \frac{\sum_i U_i^2 - \sum_i \bE(U_i^2)}{\sum_i \bE(U_i^2)}\\
                   &=\underbrace{\frac{\sum_i \phi^2 \circ F^{i-1} - N \int \phi^2}{\sum_i \bE(U_i^2)}}_{(I)}  +
                     \underbrace{\frac{\sum_i W_i - \sum_i \bE\textr{(}W_i\textr{)}}{\sum_i \bE(U_i^2)}}_{(II)} \, .
\end{align*}
Observe first that by Proposition~\ref{prop:asympVariance} and
Corollary~\ref{cor:U2}:
\begin{align*}
  \bE\left| \frac{\sum_i W_i - \sum_i \bE\textr{(}W_i\textr{)}}{\sum_i \bE(U_i^2)}
  \right| = \frac{O(N^3 L^{-\frac12})}{N(\int \phi^2 + o(1)) } = O(N^2
  L^{-\frac12}).
\end{align*}
In particular, under~\eqref{eq:N14}, term (II) above converges to $0$ in
$L^1$, hence in probability, as $L = L_n \to \infty$.
On the other hand, we will prove (I) converges to $0$ in $L^2$, hence in
probability.  In order to do this, we write:
\[
\begin{aligned}
&\quad \bE \left( \sum_i \phi^2 \circ F^{i-1} - N \int \phi^2 \right)^2   = \sum_{i, j}\left(  \int (\phi^2 \circ F^{i-1})(\phi^2 \circ F^{j-1}) - \left(\int \phi^2\right)^2  \right) \\
& = N \left( \int \phi^4 - \left(\int \phi^2\right)^2 \right) + 2 \sum_{1 \leq i < j \leq N} \left(  \int (\phi^2 \circ F^{i-1})(\phi^2 \circ F^{j-1}) - \left(\int \phi^2\right)^2  \right) .
\end{aligned}
\]
The first term on the right hand side, which corresponds to the sum
along the diagonal $i = j$, is clearly $O(N)$.  For each off-diagonal
summand $1 \leq i < j \leq N$, we apply Corollary \ref{cor:LebDoC} with
the replacements $\phi, \psi \mapsto \phi^2$ and, using~\eqref{eq:N14},
we gather
\begin{align*}
  \int (\phi^2 \circ F^{i-1})(\phi^2 \circ F^{j-1}) - \left(\int \phi^2\right)^2 %
 & = \int \phi^2 \cdot \phi^2 \circ F^{j-i} - \bigg( \int \phi^2 \bigg)^2 \\
 & = O\left( (j-i) L^{- \frac34} + L^{-\frac12}  \right) \\
 & = O(L^{- 1/2}) \, .
\end{align*}
Therefore, using once again~\eqref{eq:N14}:
\begin{align*}
  \sum_{i < j} \left(  \int (\phi^2 \circ F^{i-1})(\phi^2 \circ F^{j-1}) - \left(\int \phi^2\right)^2  \right)  = O(N^2L^{-\frac12}) = o(N^2) \, .
\end{align*}
As a result, by Proposition~\ref{prop:asympVariance}:
\[
\bE\left( \frac{ \left( \sum_i \phi^2 \circ F^{i-1} - N \int \phi^2 \right)^2}{\left( \sum_i \bE(U_i^2) \right)^2} \right) = \frac{O(N) + o(N^2) }{N^2 (\int \phi^2 + o(1))} \to 0 \, .
\]
Thus, the terms $(I)$ tend to $0$ in $L^2$, hence in probability, as
claimed.  This completes the verification of property (M3), hence the
proof of Theorem~\ref{thm:CLT}.

\section{Diffusive limit for the slow-fast system}

In this section we show how Theorem~\ref{thm:diffusion} follows from
Theorem~\ref{thm:CLT}.  Set $L = \epsilon^{-\alpha}$ and
$N(L) = N(\epsilon(L)) = \floor{\epsilon(L)^{-2}} =
\floor{L^{2/\alpha}}$; since we assume $\alpha > 8$, we have
$N(L) L^{\frac14} \to 0$ as $L \to \infty$, therefore
Theorem~\ref{thm:CLT} applies. Let $X, Y$ be independent uniformly
distributed random variables on $[0, 1]$. Since by construction
$\pi_x G_\epsilon^{i}(x, \epsilon^{1 + \alpha}y) = \pi_x F_L^{i}(x, y)$,
we have, by Theorem~\ref{thm:CLT}  applied to $\phi(x) = \sin(2 \pi x)$ that
\begin{equation}
\label{eq:limit-XY}
\begin{aligned}
  \pi_z G_\epsilon^{N(\epsilon)}(X, \epsilon^{1 + \alpha} Y) -
  \epsilon^{1 + \alpha} Y &= \epsilon\sum_{i = 0}^{N(\epsilon)-1}
  \phi \left( \pi_x G_\epsilon^i(X, \epsilon^{1 + \alpha}Y) \right)\\
  &= \epsilon \sum_{i = 0}^{N(\epsilon)-1} \phi\left( \pi_x F_L^{i}(X, Y) \right) \\
  & = \bigg(\epsilon\sqrt{N(\epsilon)} \bigg) \cdot
  \frac{1}{\sqrt{N(\epsilon)}}
  \sum_{i = 0}^{N(\epsilon)-1} \phi \left(\pi_x F_L^{i}(X, Y) \right)
  \to \Nc(0, \frac12)
\end{aligned}
\end{equation}
in distribution as $\epsilon \to 0$ (note the parenthetical term in the
third line converges to 1 as $L \to \infty$).

Recall that $Z$ is a uniformly
distributed random variable on $[a, b]$. We define
$A(\epsilon) = \epsilon^{1+ \alpha} \lceil \epsilon^{-1-\alpha}a \rceil$ and
$B(\epsilon) = \epsilon^{1 + \alpha}\floor{\epsilon^{- 1- \alpha}b}$,
and let $Z_*(\epsilon)$ be uniformly distributed on the interval
$[A(\epsilon), B(\epsilon)]$.   Notice that for $\epsilon$ sufficiently
small, $a \leq A(\epsilon) < B(\epsilon) \leq b$.

For any $i \in \Z$, the translated random variables
\begin{align*}
  \pi_z G_\epsilon^{N(\epsilon)}(X, \epsilon^{1 + \alpha}(i+ Y)) - \epsilon^{1 + \alpha}(i + Y)
\end{align*}
are all identically distributed.  As a result, the random variables
\begin{align*}
\pi_z G_\epsilon^{N(\epsilon)}(X, \epsilon^{1 + \alpha} Y) - \epsilon^{1 + \alpha} Y \text{ and }
\pi_z G_\epsilon^{N(\epsilon)}(X, Z_*(\epsilon)) - Z_*(\epsilon)
\end{align*}
are identically distributed.    %
Moreover, for any $t \in \R$, we have
\[
\begin{aligned}
  & \P\left( \pi_z G_\epsilon^{N(\epsilon)}(X, Z) - Z < t\right)  \\
  & = \P\left( \pi_z G_\epsilon^{N(\epsilon)}(X, Z) - Z < t | Z \in
    [A(\epsilon), B(\epsilon)]\right) \P(Z \in [A(\epsilon),
  B(\epsilon)]) \\
  &  \quad\quad + O\left( \P(Z \notin [A(\epsilon), B(\epsilon)] ) \right) \\
  & = \P\left( \pi_z G_\epsilon^{N(\epsilon)}(X, Z_*(\epsilon)) -
    Z_*(\epsilon) < t \right)\left( 1 - O(\epsilon^{1 + \alpha}) \right)
  + O(\epsilon^{1 + \alpha}) \, .
\end{aligned}
\]
We conclude that
\[
  \pi_z G_\epsilon^{N(\epsilon)}(X, Z) - Z, \quad \pi_z
  G_\epsilon^{N(\epsilon)}(X, Z_*(\epsilon)) - Z_*(\epsilon), \quad
  \pi_z G_\epsilon^{N(\epsilon)}(X, \epsilon^{1 + \alpha} Y) -
  \epsilon^{1 + \alpha} Y
\]
all have the same distributional limit as $\epsilon \to 0$.
Theorem~\ref{thm:diffusion} then follows from \eqref{eq:limit-XY}.
\bibliographystyle{plain}
\bibliography{clt-sm}
\end{document}